\documentclass[oneside,11pt]{amsart}

\title{Cardinal Characteristics and Countable Borel Equivalence Relations}

\author{Samuel Coskey}
\address{Samuel Coskey, Department of Mathematics, Boise State University, 1910 University Drive, Boise, ID, 83725}
\email{scoskey@nylogic.org}
\urladdr{boolesrings.org/scoskey}

\author{Scott Schneider}
\address{Scott Schneider, Department of Mathematics, University of 
Michigan, 530 Church Street, Ann Arbor, MI, 48109}
\email{sschnei@umich.edu}


\usepackage[marginratio=1:1]{geometry}
\usepackage{setspace}
\usepackage{amssymb}
\usepackage{bm}
\usepackage{mathpazo}
\usepackage{tikz}
\usepackage{diagrams}
\usepackage{rotating}
\newcommand{\downsubset}{\begin{turn}{-90}$\quad\subset\quad$\end{turn}}

\swapnumbers
\newtheorem{thm}{Theorem}[section]
\newtheorem{prop}[thm]{Proposition}
\newtheorem{lem}[thm]{Lemma}
\theoremstyle{definition}
\newtheorem{defn}[thm]{Definition}
\newtheorem{example}[thm]{Example}
\theoremstyle{remark}
\newtheorem{rem}[thm]{Remark}
\newtheorem{question}[thm]{Question}
\newtheorem*{claim*}{Claim}
\newenvironment{claimproof}{\begin{proof}[Proof of claim]}
  {\end{proof}}

\newcommand{\set}[1]{\left\{#1\right\}}
\newcommand{\seq}[1]{\left\langle#1\right\rangle}
\newcommand{\sset}{\mathsf{set}}
\newcommand{\im}{\mathop{\mathrm{im}}}
\newcommand{\Vojtas}{Vojt\'a\v{s}}
\newcommand{\id}{\mathrm{id}}

\newcommand{\fra}{\mathfrak{a}}\newcommand{\sfa}{\mathsf{a}}
\newcommand{\frb}{\mathfrak{b}}\newcommand{\sfb}{\mathsf{b}}
\newcommand{\frd}{\mathfrak{d}}\newcommand{\sfd}{\mathsf{d}}
\newcommand{\fri}{\mathfrak{i}}\newcommand{\sfi}{\mathsf{i}}
\newcommand{\frp}{\mathfrak{p}}\newcommand{\sfp}{\mathsf{p}}
\newcommand{\frr}{\mathfrak{r}}\newcommand{\sfr}{\mathsf{r}}
\newcommand{\frs}{\mathfrak{s}}\newcommand{\sfs}{\mathsf{s}}
\newcommand{\frt}{\mathfrak{t}}\newcommand{\sft}{\mathsf{t}}
\newcommand{\fru}{\mathfrak{u}}\newcommand{\sfu}{\mathsf{u}}
\newcommand{\frx}{\mathfrak{x}}
\newcommand{\sfpar}{\mathsf{par}}

\begin{document}

\begin{abstract}
  Boykin and Jackson recently introduced a property of countable Borel equivalence relations called Borel boundedness, which they showed is closely related to the union problem for hyperfinite equivalence relations.  In this paper, we introduce a family of properties of countable Borel equivalence relations which correspond to combinatorial cardinal characteristics of the continuum in the same way that Borel boundedness corresponds to the bounding number $\frb$.  We analyze some of the basic behavior of these properties, showing for instance that the property corresponding to the splitting number $\frs$ coincides with smoothness.  We then settle many of the implication relationships between the properties; these relationships turn out to be closely related to (but not the same as) the Borel Tukey ordering on cardinal characteristics.
\end{abstract}

\maketitle

\section{Introduction}

A Borel equivalence relation $E$ on the standard Borel space $X$ is called \emph{hyperfinite} if it can be written as the union of an increasing sequence of Borel equivalence relations with finite equivalence classes. Dougherty, Jackson, and Kechris developed the basic theory of hyperfinite equivalence relations in \cite{djk}, where they asked the following fundamental question which remains open:

\begin{question}[\cite{djk}]
 Is the union of an increasing sequence of hyperfinite equivalence relations hyperfinite? \label{question:unions}
\end{question}

We refer to this as the \emph{union problem}. In \cite{bj}, Boykin and Jackson introduced the notion of \emph{Borel boundedness} and showed that it is closely related to the union problem.

\begin{defn}[\cite{bj}]
  Let $E$ be a countable Borel equivalence relation on the standard Borel space $X$.  Then $E$ is \emph{Borel bounded} if for every Borel function $\phi\colon X\to\omega^\omega$, there exists a Borel function $\psi\colon X\to\omega^\omega$ satisfying $x\mathrel{E}x'\Rightarrow\psi(x)=^*\psi(x')$ and such that for all $x\in X$, $\phi(x)\leq^*\psi(x)$.
\end{defn}

Here $=^*$ and $\leq^*$ are the relations of eventual equality and eventual domination on $\omega^\omega$. Recall that a family of functions $\mathcal{F}\subset\omega^\omega$ is \emph{unbounded} if there is no function $\beta\in\omega^\omega$ such that $\alpha\leq^*\beta$ for all $\alpha\in\mathcal{F}$. The \emph{bounding number}, $\frb$, is defined to be the minimal cardinality of an unbounded family $\mathcal{F}\subset\omega^\omega$.

Of course, no countable family $\mathcal{F}=\set{\alpha_n\in\omega^\omega:n\in\omega}$ can be unbounded, since the function $\beta\in\omega^\omega$ defined by
\begin{equation}\label{eq2}
  \beta(n)=\max_{k\leq n}\alpha_k(n)
\end{equation}
eventually dominates each $\alpha_n\in\mathcal{F}$. Hence for each Borel function $\phi\colon X\to\omega^\omega$ there trivially exists an $E$-invariant function $\psi\colon X\to\omega^\omega$ such that $\phi(y)\leq^*\psi(x)$ for every $x\in X$ and $y\in [x]_E$. But for $E$ to be Borel bounded, the bounding functions $\psi(x)$ cannot depend in an essential way on the enumeration of the equivalence class $[x]_E$.

Boykin and Jackson observed that every hyperfinite equivalence relation is Borel bound\-ed, and then established the following basic link between Borel boundedness and the union problem.

\begin{thm}[\cite{bj}] \label{thm_bj}
If $E$ is the union of an increasing sequence of hyperfinite equivalence relations and $E$ is Borel bounded, then $E$ is hyperfinite.
\end{thm}

What they left open, however, in addition to the union problem itself, is the following important question.

\begin{question}
  Is Borel boundedness equivalent to hyperfiniteness?
\end{question}

There is no known example of a non-hyperfinite countable Borel equivalence relation that is Borel bounded, and the only known
examples of countable Borel equivalence relations that are \emph{not} Borel bounded have been established by Thomas under the additional assumption of Martin's Conjecture on degree invariant Borel maps \cite{mc}.

After seeing Theorem~\ref{thm_bj}, Thomas asked whether other cardinal characteristics could be used in a role similar to that played by $\frb$ in the definition of Borel boundedness. To explain, suppose that $E$ is a countable Borel equivalence relation on the standard Borel space $X$. Many cardinal characteristics of the continuum can be defined as the minimal cardinality of a subset of $\omega^\omega$ (or of $\mathcal{P}(\omega)$, or $[\omega]^\omega$) having some given combinatorial property $P$. Since each such cardinal is uncountable, it will trivially be the case that for every Borel function $\phi\colon X\to\omega^\omega$ there is an $E$-invariant function $\psi\colon X\to\omega^\omega$ such that for each $x\in X$, $\psi(x)$ witnesses the fact that the countable family $\phi([x]_E)$ does not have property $P$. However, if we require $\psi$ to be a Borel function that does not depend essentially on the representative in $[x]_E$, then such a function may or may not exist, depending on $E$.  Thus for each cardinal characteristic whose definition fits our framework, we obtain a new combinatorial property of equivalence relations that corresponds to the given cardinal in the same way that Borel boundedness corresponds to $\frb$. Our goal in this paper is to introduce and investigate these combinatorial properties.

This paper is organized as follows. We first recall some basic facts about countable Borel equivalence relations in Section~2, and then in Section~3 we consider the union problem and its connection to Borel boundedness in greater detail. In Section~4, we introduce a slew of combinatorial properties of countable Borel equivalence relations that are derived from familiar cardinal characteristics of the continuum. For the sake of the general theory, we propose a slight alteration in the definition of Borel boundedness that we show to be equivalent to the one introduced in \cite{bj}. In Section~5, we situate the preceding discussion in the abstract setting of relations and morphisms as developed by \Vojtas\ and Blass, and prove some general results concerning the properties that correspond to the so-called ``tame" cardinal characteristics.  We also show that Thomas's argument that Martin's Conjecture implies there exist non-hyperfinite, non-Borel bounded equivalence relations can be generalized to a large class of our combinatorial properties. In Section~6, we discuss the special case of the splitting number $\frs$, and show that its corresponding property coincides with smoothness. Finally, in Section~7 we establish a diagram of implications.

We wish to thank Simon Thomas for posing the questions that led to this paper.

\section{Preliminaries}

In this section, we recall some basic facts and definitions from the theory of Borel equivalence relations. For a more complete resource on the subject, we refer the reader to \cite{gao}.

A \emph{standard Borel space} is a measurable space $(X,\mathcal{B})$ such that $\mathcal{B}$ arises as the Borel $\sigma$-algebra of some Polish topology on $X$. Here a topological space is \emph{Polish} if it admits a complete, separable metric. For example, Cantor space $2^\omega$ and Baire space $\omega^\omega$ are Polish, as is $\mathbb R$. 
The set $[\omega]^\omega$ of infinite subsets of $\omega$ can be viewed as a Borel subset of $2^\omega$, and hence as a standard Borel space in its own right. The appropriate notion of isomorphism in the context of standard Borel spaces is bimeasurable bijection, which we call \emph{Borel isomorphism}. By a classical result, any two uncountable standard Borel spaces are Borel isomorphic. This will allow us to view the standard Borel spaces $\mathbb R$, $2^\omega$, $\omega^\omega$, and $[\omega]^\omega$ as equivalent, so that we may work on whichever is
most convenient.

An equivalence relation $E$ on the standard Borel space $X$ is called \emph{Borel} if $E$ is Borel as a subset of $X\times X$. The Borel equivalence relation $E$ is \emph{countable} if each of its equivalence classes is countable, \emph{finite} if each $E$-class is finite, and \emph{hyperfinite} if it can be expressed as the union $E=\cup_nF_n$ of an increasing sequence of finite Borel equivalence relations $F_n$.

If $E$ and $F$ are equivalence relations on standard Borel spaces $X$ and $Y$, a Borel function $f\colon X\to Y$ is said to be a \emph{Borel homomorphism} from $E$ to $F$, written $f\colon E\to F$, if for all $x,x'\in X$, $x\mathrel{E}x'$ implies $f(x)\mathrel{F}f(x')$. If $f$ has the stronger property that
\[
x\mathrel{E}x' \ \iff \ f(x)\mathrel{F}f(x')
\]
then $f$ is called a \emph{Borel reduction} from $E$ to $F$. If there exists a Borel reduction from $E$ to $F$, then we say that $E$ is \emph{Borel reducible} to $F$ and write $E\leq_BF$. If $E\leq_BF$ and $F\leq_BE$, then we say that $E$ and $F$ are \emph{Borel bireducible} and write $E\sim_BF$.


A Borel equivalence relation $E$ is said to be \emph{smooth} if there is a standard Borel space $Y$ such that $E\leq_B\Delta(Y)$, where $\Delta(Y)$ denotes the identity relation on $Y$. If $E$ is a countable Borel equivalence relation, then $E$ is smooth if and only if $E$ admits a Borel transversal, \emph{i.e.}, a Borel set $B\subset X$ such that $|B\cap[x]_E|=1$ for every $x\in X$. Every finite Borel equivalence relation is smooth.

Many important countable Borel equivalence relations arise naturally from the action of a countable group. If $\Gamma$ is a countable discrete group acting on the standard Borel space $X$, we write $E^X_{\Gamma}$ for the induced orbit equivalence relation defined by
\[
x\mathrel{E_\Gamma^X}y\iff(\exists\gamma\in\Gamma)\;y=\gamma x\;.
\]
If the action of $\Gamma$ on $X$ is Borel (equivalently, each $\gamma\in\Gamma$ induces a Borel map $x\mapsto \gamma x$) then $E^X_{\Gamma}$ will be a Borel equivalence relation. In fact, by a remarkable result of Feldman and Moore, if $E$ is an arbitrary countable Borel equivalence relation on the standard Borel space $X$, then there exists a countable group $\Gamma$ and a Borel action of $\Gamma$ on $X$ such that $E=E_\Gamma^X$. We shall make frequent use of this representation theorem for countable Borel equivalence relations.

If $X$ is any one of $2^\omega$, $\omega^\omega$, or $[\omega]^\omega$, we write $E_0(X)$ for the eventual equality relation on $X$. All three of these relations are Borel bireducible with one another, so when there is no danger of confusion or the domain is not important we write simply $E_0$ or $=^*$. The relation $E_0$ is the unique non-smooth hyperfinite Borel equivalence relation up to Borel bireducibility, and by the Glimm-Effros dichotomy \cite{hkl}, if $E$ is any nonsmooth Borel equivalence relation then $E_0\leq_BE$. For a survey of the general theory of hyperfinite equivalence relations, see \cite{djk} or \cite{jkl}.


We close this section with an important consequence of the Lusin-Novikov uniformization theorem (see \cite[18.10]{kechris}) that will occur frequently in our arguments.

\begin{prop} \label{prop:section}
  Suppose $X$ and $Y$ are standard Borel spaces, and suppose that $f\colon X\to Y$ is a countable-to-one Borel function. Then $\im(f)$ is a Borel subset of $Y$, and there exists a Borel function $\sigma\colon\im(f)\to X$ such that $f\circ\sigma=\id_{\im(f)}$.
\end{prop}

\begin{rem}
  The Borel function $\sigma$ given in Proposition~\ref{prop:section} is called a \emph{Borel section} for $f$. If $E$ and $F$ are countable Borel equivalence relations with $E\leq_BF$, then any Borel reduction $f$ from $E$ to $F$ is countable-to-one, and hence admits a Borel section $\sigma\colon\im(f)\to X$.
\end{rem}

\section{Borel boundedness and the union problem}

Question~\ref{question:unions} is perhaps the most basic open problem in the study of hyperfinite equivalence relations. In this
section we give an ``honest attempt" to answer it, and observe how such an attempt leads naturally to the notion of Borel boundedness.

To begin, suppose that $E=\bigcup_nD_n$ is the union of the increasing sequence of hyperfinite equivalence relations $D_n$, where for each $n$, $D_n=\bigcup_mE_n^m$ is the union of the increasing sequence of finite equivalence relations $E_n^m$. This can be pictured as the infinite grid in Figure~\ref{fig:unions}. Further assume that $E$ is the orbit equivalence relation arising from the Borel action of the countable group $\Gamma=\set{\gamma_i:i\in\omega}$, with $\gamma_0=\id$.

\begin{figure}
  \begin{diagram}[h=.5em,w=2em]
    D_0 & = & E_0^0 & \cup & E_0^1 & \cup & E_0^2 & \cup & \\
    \downsubset & & & & & & & & \\
    D_1 & = & E_1^0 & \cup & E_1^1 & \cup & E_1^2 & \cup & \cdots \\
    \downsubset & & & & & & \\
    D_2 & = & E_2^0 & \cup & E_2^1 & \cup & E_2^2 & \cup & \\
    \vdots & & & & & & \\
  \end{diagram}
  \caption{An increasing union of hyperfinite equivalence relations  
  \label{fig:unions}}
\end{figure}

As a Borel subset of $X\times X$, $E$ is itself a standard Borel space. Define the Borel function $\chi_E\colon E\to\omega^\omega$ by 
\[
\chi_E(x,y)(n)=
  \begin{cases}
  \text{the least $m$ such that $x\mathrel{E_n^m}y$}
                   & \text{if $x\mathrel{E}_ny$\;;} \\
  0                & \text{otherwise}\;.
  \end{cases}
\]
Then $\chi_E$ records exactly when a pair of elements $(x,y)\in D_n$ becomes equivalent in the union $D_n=\bigcup_m E_n^m$.

Now, let us attempt to solve the union problem by expressing $E$ as the union $E=\cup_kF_k$ of an increasing sequence of finite Borel equivalence relations $F_k$. Naively, one might simply try to ``choose the right sequence'' through Figure~\ref{fig:unions} and set each $F_k$ equal to an appropriate $E_n^m$, moving down and to the right as $k$ increases. Of course, this is bound to fail, since there is no reason a relation $E_{n}^{m}$ in row $n$ should contain some $E_{n'}^{m'}$ in row $n'$ just because $n>n'$. In fact, it is not too difficult to construct equivalence relations $E^m_n$ as in Figure~\ref{fig:unions} so that for all $n,n',m$, if $n'\ne n$ then $E_n^m\not\supset E_{n'}^0$.

To ensure that the $F_k$ are increasing, as a next step one might try taking them to be unions or intersections of the $E_n^m$. Of
course, the union of two equivalence relations need not be transitive, and there is no reason for the transitive closure of the union of two finite equivalence relations to be finite, so we are led to take intersections of the $E_n^m$. As each row is increasing, we need only take one from each row. Since the union of the $F_k$'s must exhaust $E$, we should start deleting rows from the intersection as $k$ increases. This suggests that we let
\[
F_k=\bigcap_{n\geq k}E_n^{\psi(n)}
\]
for some sequence of choices $\psi\in\omega^\omega$.

All that is left now is to make sure that the union exhausts $E$. For this we need precisely the following condition on $\psi$: for all
$(x,y)\in E$, there exists $n\in\omega$ such that for every $k\geq n$, $\psi(k)\geq\chi_E(x,y)(k)$. In other words, we need 
\[
(\forall z\in E)\; \chi_E(z)\leq^*\psi\;.
\]
Since we cannot expect a single $\psi$ to eventually dominate continuum many  functions, we should allow $\psi=\psi([x]_E)$ to depend on the equivalence classes. Since we want the $F_k$ to be Borel, this dependence will have to be Borel as well. Thus we require a Borel, $E$-invariant function $\psi\colon X\to\omega^\omega$ such that for each $(x,y)\in E$, 
\[
\chi_E(x,y)\leq^*\psi(x)=\psi(y)\;.
\]
If we define $\phi_0\colon X\to\omega^\omega$ by
\[
\phi_0(x)(n)=\max_{i\leq n}\{\chi_E(x,\gamma_ix)\}\;,
\]
then clearly $\chi_E(x,y)\leq^*\phi_0(x)$ for all $x\in X$, and so it will suffice to ask that for each $x\in X$,
\begin{equation} \label{eq1}
  \phi_0(x)\leq^*\psi(x)\;.
\end{equation}
If we ask $\psi$ to eventually dominate \emph{every} Borel function $\phi\colon X\to\omega^\omega$, we arrive at the definition of \emph{invariant Borel boundedness}.

\begin{defn} \label{df:strongbb}
  Let $E$ be a countable Borel equivalence relation on the standard Borel space $X$. Then $E$ is \emph{invariantly Borel bounded} if for every Borel function $\phi\colon X\to\omega^\omega$, there exists an $E$-invariant Borel function $\psi\colon X\to\omega^\omega$ such that $\phi(x)\leq^*\psi(x)$ for all $x\in X$.
\end{defn}

This property is too strong, however, as it is easily seen to be equivalent to smoothness.

\begin{prop} \label{prop:invar}
  Let $E$ be a countable Borel equivalence relation on the standard Borel space $X$. Then $E$ is smooth if and only if $E$ is invariantly Borel bounded.
\end{prop}

\begin{proof}
  Let $E=E^\Gamma_X$ be the orbit equivalence relation arising from the Borel action of the countable group   $\Gamma=\set{\gamma_i:i\in\omega}$ on $X$, where $\gamma_0=\id$. Suppose that $E$ is smooth, and let $B\subset X$ be
  a Borel transversal for $E$. Define the Borel function $\sigma\colon X\to X$ so that for all $x\in X$, $\sigma(x)$ is the unique
  element $y\in B$ such that $x\mathrel{E}y$. Let $\phi\colon X\to\omega^\omega$ be an arbitrary Borel function. Then we may
  define the function $\psi\colon X\to\omega^\omega$ by 
\[
\psi(x)(n)=\max_{i\leq n}\set{\phi(\gamma_i\sigma(x))(n)}\;.
\]
Clearly $\psi$ is Borel and $E$-invariant, and $\phi(x)\leq^*\psi(x)$ for all $x\in X$.

For the converse, we show that $E_0$ is \emph{not} invariantly Borel bounded; the result will then follow from the fact, noted below in the remark immediately following Proposition~\ref{prop:DownwardClosure}, that invariant Borel boundedness is closed downward under Borel reducibility. Thus suppose for contradiction that $E_0$ is invariantly Borel bounded. Identify each $x\in 2^\omega$ with the corresponding subset of $\omega$, and define $\tau(x)$ to be the increasing enumeration of $x$ if $x$ is infinite, and constantly zero otherwise, so that $\tau(x)\in\omega^\omega$. Let $\psi\colon2^\omega\to\omega^\omega$ be an $E_0$-invariant Borel function such that $\tau(x)\leq^*\psi(x)$ for all $x\in 2^\omega$. Let $D\subset 2^\omega$ be a comeager subset on which $\psi$ is continuous, so that
  \[
  \hat{D}=\bigcap_{\gamma\in\Gamma}\gamma D
  \]
  is comeager and $E_0$-invariant. Fix $x_0\in\hat{D}$. Since $[x_0]_{E_0}$ is dense and $\psi$ is constant on $[x_0]_{E_0}$, by
  continuity of $\psi$ we have $\psi(x)=\psi(x_0)$ for all $x\in\hat{D}$. However, the set 
  \[
  \set{x\in 2^\omega:\tau(x)\not\leq^*\psi(x_0)}
  \]
  is comeager, a contradiction.
\end{proof}

In Equation~\ref{eq1}, we required only that $\psi(x)$ \emph{eventually} dominate $\phi_0(x)$ for each $x\in X$, so in fact there is no reason to insist on $\psi$ being $E$-invariant; rather, it will suffice to have $\psi(x)=^*\psi(y)$ whenever $x\mathrel{E}y$, and hence it is more natural to ask for $\psi$ to be quasi-invariant. This leads to the definition introduced by Boykin and Jackson in \cite{bj}.

\begin{defn} \label{df:bb}
  Let $E$ be a countable Borel equivalence relation on the standard Borel space $X$. Then $E$ is \emph{Borel bounded} if for every Borel function $\phi\colon X\to\omega^\omega$, there exists a Borel homomorphism $\psi\colon E\to E_0(\omega^\omega)$ such that for all $x\in X$, $\phi(x)\leq^*\psi(x)$.
\end{defn}

This definition is nontrivial, since for instance any hyperfinite equivalence relation is Borel bounded.  Indeed, suppose that $E=\bigcup_nF_n$ is the union of the increasing sequence of finite Borel equivalence relations $F_n$. 
Then given any Borel function $\phi\colon X\to\omega^\omega$, we can define
\[
\psi(x)(n)=\max\set{\phi(y)(n):y\mathrel{F_n}x}\;,
\]
so that $\psi$ is a Borel homomorphism $E\to E_0(\omega^\omega)$ such that $\phi(x)\leq^*\psi(x)$ for all $x\in X$.

As we saw above, Borel boundedness is tailor-made for obtaining Theorem~\ref{thm_bj}, and we now present the proof.

\begin{proof}[Proof of Theorem~\ref{thm_bj}]
Let $E=\bigcup_nD_n$ be an increasing union, where for each $n$, $D_n=\bigcup_mE_n^m$ is the increasing union of finite Borel equivalence relations $E_n^m$.  Also, let $E$ be the orbit equivalence relation induced by the action of the countable group $\Gamma=\set{\gamma_i:i\in\omega}$, where $\gamma_0=\id$.

  Now define $\phi_0$ as in Equation~\ref{eq1}, so that for each $x$ and $n$, $\phi_0(x)(n)$ is the least $m$ such that:
  \begin{itemize}
  \item whenever $y\in\set{\gamma_0x,\ldots,\gamma_nx}$ and $y\mathrel{D_n}x$, then in fact $y\mathrel{E}_n^mx$.
  \end{itemize}
 Since $E$ is Borel bounded, there exists a Borel homomorphism $\psi\colon E\to E_0(\omega^\omega)$ such that for all $x\in X$, $\phi_0(x)\leq^*\psi(x)$.  We may therefore define $F_n$ by:
   \begin{itemize}
  \item $x\mathrel{F}_ny$ iff for all $k\geq n$, we have $\psi(x)(k)=\psi(y)(k)$ and $x\mathrel{E}_k^{\psi(x)(k)}y$.
  \end{itemize}
  (The ``for all $k\geq n$'' is needed to make the $F_n$ increasing, the ``$\psi(x)(k)=\psi(y)(k)$'' is needed to make $F_n$ symmetric, the ``$E_k$'' is needed to make $F_n$ finite, and the ``$\psi(x)(k)$'' is needed to ensure the $F_n$ will exhaust $E$.)

It is clear that $F_n$ is an increasing sequence of finite equivalence relations contained in $E$.  The last thing to check is that $E=\bigcup_nF_n$.  Indeed, if $x\mathrel{E}y$, then write $y=\gamma_ix$ and choose some $n>i$ such that for all $k\geq n$, we have $x\mathrel{E_k}y$ and
  \[
  \max\set{\phi_0(x)(k),\,\phi_0(y)(k)} \ \leq \ \psi(x)(k) \ = \ \psi(y)(k)\;.
  \]
  Then by definition of $\phi_0$, for each $k\geq n$ we have $x\mathrel{E}_k^{\phi_0(x)(k)}y$ and hence $x\mathrel{E}_k^{\psi(x)(k)}y$.
\end{proof}

As mentioned in the introduction, it is not known whether there exist Borel bounded countable Borel equivalence relations that are not hyperfinite.

We conclude this section with a proof of the fact that Borel boundedness is closed downward under Borel reducibility. This result is Lemma~10 of \cite{bj}, but we present a proof that is designed to motivate our subsequent discussion of several analogous combinatorial properties.

\begin{prop}
  Let $E$ and $F$ be countable Borel equivalence relations on the standard Borel spaces $X$ and $Y$, respectively. If $E\leq_BF$ and
  $F$ is Borel bounded, then $E$ is Borel bounded. \label{prop:DownwardClosure}
\end{prop}

\begin{proof}
  Let $\phi\colon X\to\omega^\omega$ be any Borel function. Suppose that $f\colon X\to Y$ is a Borel reduction from $E$ to $F$, and define the equivalence relation $E'\subset E$ on $X$ by $x\mathrel{E'}y$ iff $f(x)=f(y)$. Then $E'$ is smooth, and therefore
  Borel bounded. Let $\phi'\colon E\to E_0(\omega^\omega)$ be a Borel homomorphism such that for all $x\in X$,
  $\phi(x)\leq^*\phi'(x)$. Also let $B\subset X$ be a Borel transversal for $E'$, with $\sigma\colon\im(f)\to X$ a Borel function such that $f\circ\sigma=\id_{\im(f)}$ (see figure~\ref{fig_reduction}).
  \begin{figure}
    \begin{tikzpicture}
      \draw (0,0) rectangle (3,4);
      \draw (0,0) node [anchor=north west] {$E$-class};
      \draw (5,0) rectangle (6,5);
      \draw (5,0) node [anchor=north west] {$F$-class};
      \draw[dotted] (5,4) -- (6,4);
      \foreach \y in {.5,2.5,3,3.5} {
        \draw (.3,\y)--(2.7,\y);
        \draw[fill] (5.5,\y) circle (2pt);
        \draw[fill] (1.5+rand,\y) circle (2pt);
      }
      \draw (1.5,1.5) node {$\vdots$};
      \draw (5.5,1.5) node {$\vdots$};
      \draw[->] (3.3,2.4) -- node[above]{$f$} (4.7,2.4);
      \draw[<-] (3.3,1.6) -- node[below]{$\sigma$} (4.7,1.6);
    \end{tikzpicture}
    \caption{$\sigma$ gives us a base point with which to order each
      fiber of $f$.\label{fig_reduction}}
  \end{figure}
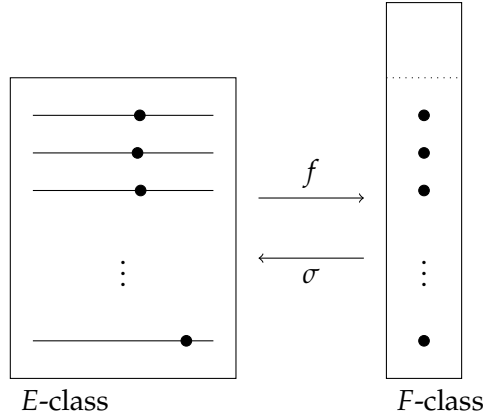
  Now define the Borel function $\tilde{\phi}\colon
  Y\to\omega^\omega$ by
  \[\tilde{\phi}(y)(n)=\begin{cases}
    \phi'(\sigma(y))(n) & \text{if $y\in\im(f)$}\;; \\
    0 & \text{otherwise}\;.
    \end{cases}
  \]
Using the fact that $F$ is Borel bounded, let $\tilde{\psi}\colon F\to E_0(\omega^\omega)$ be a Borel homomorphism such that for all $y\in Y$, $\tilde{\phi}(y)\leq^*\tilde{\psi}(y)$. Finally, let $\psi=\tilde{\psi}\circ f$. Then $\psi$ is a Borel homomorphism from $E$ to $E_0(\omega^\omega)$ such that for all $x\in X$, $\phi(x)\leq^*\psi(x)$.
\end{proof}

In the proof of Proposition~\ref{prop:invar}, we required that \emph{invariant} Borel boundedness is also closed downward under Borel reducibility.  Indeed, this follows using the same argument, since if $F$ is invariantly Borel bounded, then the Borel function $\tilde{\psi}$ in the proof of Proposition~\ref{prop:DownwardClosure} can be chosen to be $F$-invariant, in which case $\psi$ will be $E$-invariant.

\section{Combinatorial properties and cardinal characteristics}

In the definition of Borel boundedness, the Borel function $\phi$ assigns a countable family of elements of $\omega^\omega$ to each $E$-class $[x]_E$.  Since no countable family is unbounded, there is always a witness $\psi(x)$ which bounds $\phi([x]_E)$.  Borel boundedness means that this witness can be chosen in an explicit and quasi-invariant manner that does not depend on an enumeration of $[x]_E$.  An analogous definition can be made in which unbounded families are replaced by other types of families which appear in the study of cardinal characteristics of the continuum: splitting families, maximal almost disjoint families, ultrafilter bases, and so on.  We will now introduce combinatorial properties that correspond to these other cardinal characteristics in the same way that Borel boundedness corresponds to the bounding number $\frb$.

We focus in this section on a few of the most natural combinatorial cardinal characteristics.  The relationship between the sizes of these cardinals is described visually by (a subset of) the so-called van~Douwen diagram, which appears in Figure~\ref{fig_vand}. The article \cite{blass} gives a full discussion of these cardinals and their relationships. As motivating examples we consider the splitting number $\frs$ and the pseudo-intersection number $\frp$.

\begin{figure}
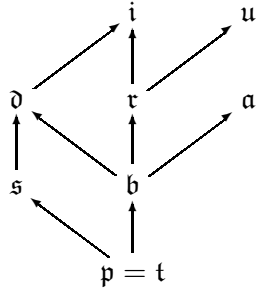

\begin{displaymath}\
\begin{diagram}[h=1.5em,w=2em]
     &       & \fri &       & \fru \\
     & \ruTo & \uTo & \ruTo &      \\
\frd &       & \frr &       & \fra \\
\uTo & \luTo & \uTo & \ruTo &      \\
\frs &       & \frb &       &      \\
     & \luTo & \uTo &       &      \\
     &       & \frp=\frt &       &      \\
\end{diagram}
\end{displaymath}
\caption{Size relationships among several cardinal characterisitcs\label{fig_vand}}
\end{figure}

Given a subset $A\subset\omega$, write $A^c=\omega\smallsetminus A$. Given sets $A,B\subset\omega$, we say that $A$ \emph{splits} $B$ if $|A\cap B|=|A^c\cap B|=\aleph_0$. A family $\mathcal{S}\subset[\omega]^\omega$ of infinite subsets of $\omega$ is a \emph{splitting family} if for every infinite set $B\subset\omega$ there exists $A\in\mathcal{S}$ such that $A$ splits $B$. The \emph{splitting number} $\frs$ is defined to be the minimum cardinality of a splitting family.

A family $\mathcal{F}\subset [\omega]^\omega$ of infinite subsets of $\omega$ is \emph{centered} if every finite subfamily of $\mathcal{F}$ has infinite intersection. The infinite set $A\subset\omega$ is said to be a \emph{pseudo-intersection} of the centered family $\mathcal{F}$ if $A\subset^*B$ for every $B\in\mathcal{F}$. The pseudo-intersection number $\frp$ is defined to be the minimum cardinality of a centered family with no pseudo-intersection.

As it will be relevant later, we briefly sketch a proof of the fact that no countable family of subsets of $\omega$ can be a splitting
family.

\begin{prop} \label{prop:Suncountable}
  $\aleph_0<\frs$.
\end{prop}

\begin{proof}
  Let $\set{A_n:n\in\omega}$ be a countable family of subsets of $\omega$. Given some nonprincipal ultrafilter $\mathcal{U}$ on $\omega$, we can set $B_{-1}=\omega$ and then inductively define $B_{n+1}$ to be whichever of the sets $B_n\cap A_n$, $B_n\cap A_n^c$ is in $\mathcal{U}$. Then each $B_n$ is infinite, and if we inductively choose $b_{n+1}\in B_{n+1}$ distinct from $b_0,\ldots,b_n$, then $\set{b_n:n\in\omega}$ is not split by any $A_n$.
\end{proof}

It follows that if $E$ is a Borel equivalence relation on the standard Borel space $X$, then for each Borel function $\phi\colon X\to[\omega]^\omega$, there trivially exists an $E$-invariant function $\psi\colon X\to[\omega]^\omega$ such that for each $x\in X$, $\psi(x)$ witnesses the fact that $\phi([x]_E)$ is not a splitting family; \emph{i.e.}, such that for each $x\in X$, $\psi(x)\subset^*\phi(x)$ or $\psi(x)\subset^*\phi(x)^c$. In analogy with Borel boundedness, we might therefore call $E$ \emph{Borel non-splitting} if such a function $\psi$ can be chosen in an explicit and quasi-invariant fashion.

\begin{defn}[temporary] \label{df:tempS}
  Let $E$ be a countable Borel equivalence relation on the standard Borel space $X$. Then $E$ is \emph{Borel non-splitting} if for every Borel function $\phi\colon X\to[\omega]^\omega$, there exists a Borel homomorphism $\psi\colon E\to E_0([\omega]^\omega)$ such that for all $x\in X$, $\psi(x)\subset^*\phi(x)$ or $\psi(x)\subset^*\phi(x)^c$.
\end{defn}

Similarly, every countable centered family $\mathcal{F}\subset[\omega]^\omega$ has a pseudo-intersection; hence $\aleph_0<\frp$, which suggests the following definition.

\begin{defn}[temporary] \label{df:tempP}
  Let $E$ be a countable Borel equivalence relation on the standard Borel space $X$. Then $E$ is \emph{Borel pseudo-intersecting} if for every Borel function $\phi\colon X\to[\omega]^\omega$ such that the family $\set{\phi(y):y\mathrel{E}x}$ is centered for all $x\in X$, there exists a Borel homomorphism $\psi\colon E\to E_0([\omega]^\omega)$ such that for all $x\in X$,
  $\psi(x)\subset^*\phi(x)$.
\end{defn}

If these are to be reasonable properties of countable Borel equivalence relations, then it is highly desirable for them to be closed downward under Borel reducibility, or at least to be $\sim_B$-invariant. However, two different problems arise if one attempts to prove the analogue of Proposition~\ref{prop:DownwardClosure} for the notions just defined.

Notice that the proof of Proposition~\ref{prop:DownwardClosure} involved two separate diagonalizations. First, each fiber $f^{-1}(\set{y})$ yielded a countable family $\set{\phi(x):f(x)=y}\subset\omega^\omega$ that was eventually dominated by $\phi'(\sigma(y))$; then the countable family $\set{\phi'(\sigma(z)):z\in [y]_F}$ was eventually dominated by $\tilde{\psi}(y)$ and this served to eventually dominate the entire original family $\set{\phi(x):x\mathrel{E}\sigma(y)}$.  Letting $\frx$ be the cardinal under consideration, it is apparent that this argument only works if the property of ``witnessing that $\aleph_0<\frx$'' is transitive.  It clearly is in the case of $\frb$, but if each $B_n\subset\omega$ witnesses that $\set{A_n^m:m\in\omega}$ is not splitting, and if $C\subset\omega$ witnesses that $\set{B_n:n\in\omega}$ is not splitting, then nevertheless $C$ may very well be split by some set $A_n^m$.

The pseudo-intersecting property yields an even more fundamental problem. If 
\[
\set{A_n^m:m,n\in\omega}
\]
is a centered family of subsets of $\omega$, and if $B_n$ is a pseudo-intersection of $\set{A_n^m:m\in\omega}$ for each $n$, then there is no reason for $\set{B_n:n\in\omega}$ even to be centered, so there is no way to carry out a second diagonalization.

Consequently we propose the following slight adjustment in the definition of Borel boundedness, justified below by Proposition~\ref{prop:bb=prime}, so that it can properly serve as a model for the general theory. Recall that if $X$ is a standard Borel space, then we define the equivalence relation $E_\sset(X)$ on $X^\omega$ by
\[
\seq{x_n}\mathrel{E_\sset(X)}\seq{x'_n}\iff \set{x_n:n\in\omega}=\set{x'_n:n\in\omega}\;.
\]
If $X$ is clear from context, we shall often write $E_\sset$ instead of $E_\sset(X)$.

\begin{defn} \label{df:prime}
  Let $E$ be a countable Borel equivalence relation on the standard Borel space $X$. Then $E$ has property $\sfb$ ($E$ is \emph{Borel bounded}) if for every Borel homomorphism $\phi\colon X\to(\omega^\omega)^\omega$ from $E$ to $E_\sset(\omega^\omega)$, there exists a Borel homomorphism $\psi\colon X\to \omega^\omega$ from $E$ to $E_0(\omega^\omega)$ such that for all $x\in X$ and for all $n\in\omega$, $\phi(x)(n)\leq^*\psi(x)$.
\end{defn}

In the original Definition~\ref{df:bb}, the function $\phi$ assigns a countable family of functions $\phi([x]_E)\subset\omega^\omega$ to each $E$-class by associating a single one to each element in the class. In Definition~\ref{df:prime}, we give \emph{each} element $x\in X$ knowledge of the entire family of functions that $\phi$ associates to $[x]_E$. As we will see below, this slight tweak will enable us to prove downward closure under $\leq_B$ for all of the properties defined below.

\begin{prop} \label{prop:bb=prime}
  Let $E$ be a countable Borel equivalence relation on the standard Borel space $X$. Then $E$ is Borel bounded in the sense of
  Definition~\ref{df:bb} if and only if $E$ has property $\sfb$, \emph{i.e.}, is Borel bounded in the sense of Definition~\ref{df:prime}.
\end{prop}

\begin{proof}
  Let $E$ be the orbit equivalence relation arising from the Borel action of the countable group $\Gamma=\seq{\gamma_n:n\in\omega}$. Suppose $E$ has property $\sfb$, and let $\phi\colon X\to\omega^\omega$ be a Borel function. Then for each $x\in X$ and $n\in\omega$, define $\phi'(x)(n)=\phi(\gamma_nx)$, so that $\phi'$ is a Borel homomorphism from $E$ to $E_\sset(\omega^\omega)$. If $\psi$ is a Borel homomorphism from $E$ to $E_0(\omega^\omega)$ such that $\phi'(x)(n)\leq^*\psi(x)$ for all $x\in X$ and $n\in\omega$, then $\phi(x)\leq^*\psi(x)$ for all $x\in X$.

  Conversely, suppose $E$ is Borel bounded and let $\phi'\colon X\to(\omega^\omega)^\omega$ be a Borel homomorphism from $E$ to $E_\sset(\omega^\omega)$. Then for each $x\in X$ and $n\in\omega$, define
  \[
  \phi(x)(n)=\max_{k\leq n}\set{\phi'(x)(k)(n)}\;,
  \]
  so that $\phi\colon X\to\omega^\omega$ is a Borel function such that for each $x\in X$ and $n\in\omega$, $\phi'(x)(n)\leq^*\phi(x)$. Obtain a Borel homomorphism $\psi\colon E\to E_0(\omega^\omega)$ such that   $\phi(x)\leq^*\psi(x)$ for all $x\in X$. Then $\phi'(x)(n)\leq^*\psi(x)$ for each $x\in X$ and $n\in\omega$.
\end{proof}

We are now ready to introduce a zoo of combinatorial properties of countable Borel equivalence relations, each of which corresponds to a cardinal characteristic of the continuum in the same way that Borel boundedness corresponds to $\frb$.

\begin{defn}
  \label{df:list}
  Let $E$ be a countable Borel equivalence relation on the standard Borel space $X$. In each of the following, $\phi$ and $\psi$ always denote Borel homomorphisms.
\begin{list}{$\circ$}{\leftmargin=5mm\itemsep=3mm\itemindent=-2mm}
\item $E$ has property $\sfb$ ($E$ is \emph{Borel bounded}) if for every $\phi\colon E\to E_\sset(\omega^\omega)$, there exists $\psi\colon E\to E_0(\omega^\omega)$ such that for all $x\in X$ and $n\in\omega$, $\phi(x)(n)\leq^*\psi(x)$.

\item $E$ has property $\sfd$ ($E$ is \emph{Borel non-dominating}) if for every $\phi\colon E\to E_\sset(\omega^\omega)$, there exists $\psi\colon E\to E_0(\omega^\omega)$ such that for all $x\in X$ and $n\in\omega$, $\psi(x)\not\leq^*\phi(x)(n)$.

\item $E$ has property $\sfs$ ($E$ is \emph{Borel non-splitting}) if for every $\phi\colon E\to E_\sset([\omega]^\omega)$, there exists $\psi\colon E\to E_0([\omega]^\omega)$ such that for all $x\in X$ and $n\in\omega$, $\psi(x)\subset^*\phi(x)(n)$ or $\psi(x)\subset^*\phi(x)(n)^c$.

\item $E$ has property $\sfr$ ($E$ is \emph{Borel reapable}) if for every $\phi\colon E\to E_\sset([\omega]^\omega)$, there exists $\psi\colon E\to E_0([\omega]^\omega)$  such that for all $x\in X$ and $n\in\omega$, $|\psi(x)\cap\phi(x)(n)|=|\psi(x)^c\cap\phi(x)(n)|=\aleph_0$.

\item $E$ has property $\sfp$ ($E$ is \emph{Borel pseudo-intersecting}) if for every $\phi\colon E\to E_\sset([\omega]^\omega)$ such that $\set{\phi(x)(n):n\in\omega}$ is centered for each $x\in X$, there exists $\psi\colon E\to E_0([\omega]^\omega)$ such that for all $x\in X$ and $n\in\omega$, $\psi(x)\subset^*\phi(x)(n)$.

\item $E$ has property $\sft$ ($E$ is \emph{Borel tower-plugging}) if for every $\phi\colon E\to E_\sset([\omega]^\omega)$ such that $\set{\phi(x)(n):n\in\omega}$ admits a well-ordering compatible with $\subset^*$, there exists $\psi\colon E\to E_0([\omega]^\omega)$ such that for all $x\in X$ and $n\in\omega$, $\psi(x)\subset^*\phi(x)(n)$.

\item $E$ has property $\sfa$ ($E$ is \emph{Borel non-mad}) if for every $\phi\colon E\to E_\sset([\omega]^\omega)$ with the property that $\set{\phi(x)(n):n\in\omega}$ is almost disjoint for all $x\in X$, there exists $\psi\colon E\to E_0([\omega]^\omega)$ such that for all $x\in X$ and $n\in\omega$, $|\psi(x)\cap\phi(x)(n)|<\aleph_0$.

\item $E$ has property $\sfi$ ($E$ is \emph{Borel non-maximally-independent}) if for every $\phi\colon E\to E_\sset([\omega]^\omega)$ such that $\set{\phi(x)(n):n\in\omega}$ is independent for all $x\in X$, there exists $\psi\colon E\to E_0([\omega]^\omega)$ such that for all $x\in X$, the set $\set{\psi(x)}\cup\set{\phi(x)(n):n\in\omega}$ is independent and $\psi(x)\not=^*\phi(x)(n)$ for all $n\in\omega$.

\item $E$ has property $\sfu$ if for every $\phi\colon E\to E_\sset([\omega]^\omega)$ such that $\set{\phi(x)(n):n\in\omega}$ is centered for all $x\in X$, there exists $\psi\colon E\to E_0([\omega]^\omega)$ such that for all $x\in X$ and $n\in\omega$, $|\psi(x)\cap\phi(x)(n)|=|\psi(x)^c\cap\phi(x)(n)|=\aleph_0$.
\end{list}
\end{defn}

Thus we obtain combinatorial properties for each of the cardinal characteristics $\frb$, $\frd$, $\frs$, $\frr$, $\frp$, $\fra$, $\fri$, $\frt$, and $\fru$, which for convenience we denote with the same letters in a sans-serif font. We include both $\frp$ and $\frt$ even though $\frp=\frt$; these properties are defined differently and may not be the same.

It is easy to show that each of these properties is closed downward under containment and Borel reducibility. This is proved generally in the next section, but we sketch here the proof for property $\sfp$ to illustrate the motivation behind the use of $E_\sset$ in Definition~\ref{df:list}.

\begin{prop}
Let $E$ and $F$ be countable Borel equivalence relations on the standard Borel spaces $X$ and $Y$, respectively. If $F$ has property $\sfp$ and $E\leq_BF$, then $E$ has property $\sfp$.
\end{prop}

\begin{proof}
Let $E$ be the orbit equivalence relation arising from the Borel action of the countable group $\Gamma=\set{\gamma_i:i\in\omega}$ on $X$. Suppose that $f\colon X\rightarrow Y$ is a Borel reduction from $E$ to $F$, and let $\sigma\colon\im(f)\to X$ be a Borel function such that $f\circ\sigma=\id_{\im(f)}$, as depicted in Figure~\ref{fig_reduction}.  %
Note that $\im(f)$ is Borel and that $F'=F\upharpoonright\im(f)$ has property $\sfp$ (see \ref{thm:closure}(i)). %
Suppose $\phi\colon E\to E_\sset([\omega]^\omega)$ is a Borel homomorphism such that for every $x\in X$, the family $\set{\phi(x)(n):n\in\omega}$ is centered. We shall define another homomorphism $\tilde\phi\colon F'\to E_\sset([\omega]^\omega)$ with the analogous property.  First, let $n\mapsto\seq{n_0,n_1}$ denote a fixed pairing function. Define
  \[
  \tilde\phi(y)(n)=\phi(\gamma_{n_0}\sigma(y))(n_1)\;.
  \]
Then $\tilde\phi$ is a Borel homomorphism from $F'$ to $E_\sset([\omega]^\omega)$ such that $\set{\tilde{\phi}(y)(n):n\in\omega}$ is centered for each $y$, so using property $\sfp$ fix a Borel homomorphism $\tilde{\psi}\colon F'\to E_0([\omega]^\omega)$ such that $\tilde{\psi}(y)\subset^*\tilde{\phi}(y)(n)$ for all $y$ and $n$. Now define $\psi=\tilde{\psi}\circ f$. Since $\phi(x)$ and $\tilde{\phi}(f(x))$ enumerate the same families, we have $\psi(x)\subset^*\phi(x)(n)$ for all $n$.
\end{proof}

Our aim is now to establish the basic relationships between these properties, and attempt to locate them as far as possible within the hierarchy of countable Borel equivalence relations under the partial (pre)-order of Borel reducibility.

As we shall see, there is a rough correspondence between ZFC-provable inequalities among cardinal characteristics and implications among the combinatorial properties we have introduced. For example, it is immediate from the definitions that every Borel bounded relation is also Borel non-dominating, for the same reason that $\frb\leq\frd$. We express this succinctly by writing $\sfb\rightarrow\sfd$. Likewise, the fact that every tower is a centered family shows both that $\frp\leq\frt$ and also that $\sfp\rightarrow\sft$.

Figure~\ref{fig_invar} displays the basic implications between properties corresponding to cardinal characteristics, which will be proved in Section~7. A special case appears to be property $\sfs$, which we show in Section~6 to be equivalent to smoothness. At present, most of the implications in Figure~\ref{fig_invar} are simply due to ``obvious'' inequalities between cardinal characteristics. We hope that the diagram can be improved and expanded in the future.

\begin{figure}[ht]
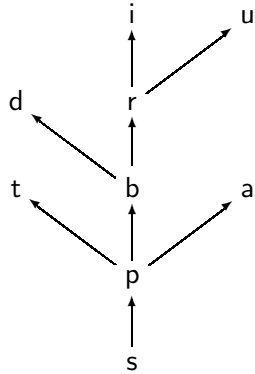

\begin{displaymath}
\begin{diagram}[h=1.5em,w=2em]
     &       & \sfi &       & \sfu \\
     &       & \uTo & \ruTo &      \\
\sfd &       & \sfr &       &      \\
     & \luTo & \uTo &       &      \\
\sft &       & \sfb &       & \sfa \\
     & \luTo & \uTo & \ruTo &      \\
     &       & \sfp &       &      \\
     &       & \uTo &       &      \\
     &       & \sfs &       &      \\
\end{diagram}
\end{displaymath}
\caption{Relationships among combinatorial properties\label{fig_invar}}
\end{figure}

It is possible that this forest of definitions will give rise to a varied and interesting family of properties.  It is also conceivable that each of these properties implies hyperfiniteness, and hence that the whole diagram collapses.  If this is the case, then we will be left with a great number of new characterizations of hyperfiniteness. In light of the challenges in this subject, if the diagram does not collapse then it may be very difficult to prove that this is the case.

We conclude this section by discussing briefly some of the motivations for studying and expanding the diagram in Figure~\ref{fig_invar}, and by considering what meaning the relationships between these properties might have for the cardinals themselves. The diagram may be regarded as being stratified into layers, though we do not know where the boundaries lie at the moment.  At the bottom of the diagram we find the characterizations of smoothness, such as property $\sfs$.  Just above that should be the characterizations of (nonsmooth) hyperfiniteness.  It is not known whether any of the properties in the diagram lie at this level.

From the perspective of the union problem, the most interesting layer is the one consisting of those properties $\mathsf{x}$ that do not imply hyperfiniteness, but additionally have the Boykin-Jackson property: if $E$ is an increasing union of hyperfinite equivalence  relations and $E$ has property $\mathsf{x}$ then $E$ is hyperfinite.  Once again, we do not know if this layer is nonempty, but it is at this layer that we have a solution for the union problem.

The top layer consists of those properties which hold of all countable Borel equivalence relations.  Once again, we do not know any (nontrivial) property which lies at this layer.  But if this layer were to overlap with the previous, then the union problem would be completely solved.

Finally, we suggest another simple meaning for our combinatorial properties. Each one concerns a family of a certain type: a dominating family, a maximal tower, a centered family with no pseudointersection, and so on. None of these families can be countable, as can usually be established by a straightforward diagonalization argument. However, the witness constructed in this diagonalization typically depends on a well-ordering of the family. The corresponding combinatorial property helps us to gauge the extent of this dependence; the higher the corresponding property lies in the diagram, the easier it is to diagonalize in a way that does not explicitly depend on the well-order.

As an example, consider the diagonalization that is used to prove that no countable family can be unbounded, compared to that used in the proof of Proposition~\ref{prop:Suncountable} to show that no countable family can be a splitting family. The construction given in the latter argument is less explicit and depends to a greater extent on the ordering of the sets $\langle A_n\rangle$ than that given in Equation~\eqref{eq2} to show that $\aleph_0<\frb$. In particular, it requires that we make a sequence of dependent and seemingly arbitrary choices, for which we enlisted the aid of a nonprincipal ultrafilter. In Section~6 we demonstrate that in fact the diagonalization cannot be carried out in a Borel and (quasi-)invariant fashion when the countable families are given by $E_0$-classes.

\section{A general framework for combinatorial properties}

Most cardinal characteristics, including those considered in Section~4, can be described abstractly as the norm of a suitable relation. See, for instance, \Vojtas~\cite{voj} and Blass \cite{blass-borel} for a development of this approach.  We now extend our discussion to this setting, and define combinatorial properties corresponding to nearly any cardinal of this type.

Following \Vojtas, let a \emph{relation} be a triple $\bm{A}=(A_-,A_+,A)$, where $A\subset A_-\times A_+$.  We think of $A_-$ as a set of ``challenges,'' $A_+$ as a set of ``responses,'' and $A$ as a ``dominating relation.''  We say that a family $\mathcal{F}\subset A_+$ is \emph{dominating} with respect to the relation $\bm{A}$ if for any challenge $x\in A_-$ there exists a
response $y\in\mathcal{F}$ such that $x\mathrel{A}y$.  Let $D(\bm{A})$ consist of all families $\mathcal{F}\subset A_+$ that are dominating with respect to $\bm{A}$. Then the \emph{dominating number} or \emph{norm} of the relation $\bm{A}$ is the cardinal characteristic
\[
\|\bm{A}\|=\min\set{|\mathcal F|:\mathcal F\in D(\bm{A})}\,.
\]
For example, the dominating number $\mathfrak d$ is the norm of the relation $(\omega^\omega,\omega^\omega,\mathord{\leq}^*)$, and the splitting number $\mathfrak s$ is the norm of the relation $([\omega]^\omega,\mathcal P(\omega),\text{``is split by''})$. Indeed, many cardinal characteristics can be expressed as the norm of a suitable relation, but a slight generalization is needed to handle certain others.  Let $\bm{A}$ be a relation, and let $\Phi$ be a property of families $\mathcal F\subset A_+$.  Then the norm of $\bm{A}$ \emph{relative to $\Phi$} is the cardinal characteristic
\[
\|\bm{A}\|_\Phi=\min\set{|\mathcal F|:\Phi(\mathcal F)\;\&\;\mathcal F\in D(\bm{A})}\;.
\]
For instance, the pseudo-intersection number $\frp$ is the norm of the relation $([\omega]^\omega,2^\omega, \not\subset^*\nolinebreak)$ relativized to the property $\Phi(\mathcal F)=$ ``$\mathcal F$ is centered.''

Cardinal characteristics that can be expressed as the relativized norm of a relation are said to be \emph{tame}.  (This terminology is usually reserved for the case when additional definability constraints are imposed on $\bm{A}$ and $\Phi$.  At present we shall not have need of any further precision, but see Appendix~B of \cite{zap} for a discussion.)

For our purposes, we consider only relations on spaces with an eventual equality relation. Specifically, a relation $\bm{A}$ is said to be \emph{invariant} if both $A_-$ and $A_+$ are subsets of either $2^\omega$ or $\omega^\omega$, and whenever $x\mathrel{E_0(A_-)}x'$ and $y\mathrel{E_0(A_+)}y'$, we have
\[
x\mathrel{A}y\iff x'\mathrel{A}y'\,.
\]
Now we are ready to define the combinatorial property corresponding to the cardinal characteristic $\|\bm A\|_\Phi$ as follows.


\begin{defn}
  Let $\bm{A}$ be an invariant relation, let $\Phi$ be a property of families $\mathcal{F}\subset A_+$, and suppose that $E$ is a countable Borel equivalence relation on the standard Borel space $X$. Then $E$ is said to be \emph{Borel non-$(\bm{A},\Phi)$} if for every Borel homomorphism $\phi\colon E\to E_\sset(A_+)$ such that the family $\set{\phi(x)(n):n\in\omega}$ has property $\Phi$ for each $x\in X$, there exists a Borel homomorphism $\psi\colon E\to E_0(A_-)$ such that for all $x\in X$ and $n\in\omega$,
\[
\neg(\psi(x)\mathrel{A}\phi(x)(n))\;.
\]
For the ``non-relativized'' case where $\Phi$ holds of all families, we just say \emph{Borel non-$\bm{A}$}.
\end{defn}


For example, the property corresponding to $\frd=\|(\omega^\omega,\omega^\omega,\leq^*)\|$ is the Borel non-dominating property, $\sfd$.  Similarly, the Borel pseudo-intersecting property $\mathsf{p}$ is the property corresponding to the relation $([\omega]^\omega,\mathcal{P}(\omega),\not\subset^*)$ together with $\Phi(\mathcal F)=$ ``$\mathcal F$ is centered.'' In
fact, one easily checks that each of the nine properties introduced in Definition~\ref{df:list} is derived from a tame cardinal characteristic. 

We now establish some basic closure results. Recall that if $E\subset F$ are countable Borel equivalence relations on the standard Borel space $X$, then a Borel set $B\subset X$ is \emph{full} for $E$ if $B$ intersects each $E$-class, and that $F$ is \emph{smooth over} $E$ if there is a Borel homomorphism $f\colon F\to E$ such that $x\mathrel{F}f(x)$ for all $x\in X$.

\begin{thm} \label{thm:closure}
Let $E$, $F$ be countable Borel equivalence relations on the standard Borel spaces $X$, $Y$, respectively, and suppose that $\bm{A}$ is an invariant relation, with $\Phi$ a property of families $\mathcal{F}\subset A_+$. 
\begin{enumerate}
  \item[(i)] If $E$ is Borel non-$(\bm{A},\Phi)$ and $B\subset X$ is Borel, then $E\mathbin{\upharpoonright}B$ is Borel non-$(\bm{A},\Phi)$.
  \item[(ii)] If $E\leq_BF$ and $F$ is Borel non-$(\bm{A},\Phi)$, then $E$ is Borel non-$(\bm{A},\Phi)$.
  \item[(iii)] If $B\subset X$ is a Borel subset that is full for $E$ and if $E\mathbin{\upharpoonright}B$ is Borel non-$(\bm{A},\Phi)$, then $E$ is Borel non-$(\bm{A},\Phi)$.
  \item[(iv)] If $X=Y$, $E\subset F$, $F$ is smooth over $E$, and $E$ is Borel non-$(\bm{A},\Phi)$, then $F$ is Borel non-$(\bm{A},\Phi)$.
  \item[(v)] If $X=Y$, $E\subset F$, $F$ is Borel non-$(\bm{A},\Phi)$, and the property $\Phi$ is closed under unions, then $E$ is Borel non-$(\bm{A},\Phi)$.
  \end{enumerate}
\end{thm}

\begin{proof}
  Throughout the proofs of (i)--(v) we fix a pairing function $n\mapsto\langle n_0,n_1\rangle$ of $\omega$ onto $\omega^2$.

(i) Fix $x_0\in B$ (if $B=\emptyset$ there is nothing to prove), and let $E$ be the orbit equivalence relation arising from the Borel action of the group $\Gamma=\set{\gamma_i:i\in\omega}$ on $X$, where $\gamma_0=\id$. Suppose that $\phi\colon E\mathbin{\upharpoonright}B\to E_\sset(A_+)$ is a Borel homomorphism such that for all $x\in B$, $\{\phi(x)(n):n\in\omega\}$ has property $\Phi$. Let $[B]=\set{x\in X:(\exists y\in B)\,x\mathrel{E}y}$ be the $E$-saturation of $B$, and define the function $\sigma\colon [B]\to B$ by
\[
\sigma(x)=\gamma_kx,\text{ where $k$ is least such that $\gamma_kx\in B$}\;.
\]
Define the function $\tilde{\phi}\colon X\to(A_+)^\omega$ by
\[
\tilde{\phi}(x)(n)= \begin{cases} \phi(\sigma(x))(n) & \text{if $x\in [B]$\;;} \\ \phi(x_0)(n) & \text{otherwise\;.} \end{cases}
\]
Then $\tilde{\phi}$ is a Borel homomorphism $E\to E_\sset(A_+)$ such that for all $x\in X$, $\{\tilde{\phi}(x)(n):n\in\omega\}$ has property $\Phi$. By hypothesis obtain a Borel homomorphism $\psi\colon E\to E_0(A_-)$ such that for all $x\in X$ and $n\in\omega$, $\neg(\psi(x)\mathrel{A}\tilde{\phi}(x)(n))$. The restriction of $\psi$ to $B$ is as desired.

(ii) Suppose that $E$ is the orbit equivalence relation arising from the Borel action of the group $\Gamma=\set{\gamma_i:i\in\omega}$ on $X$. Let $f\colon X\to Y$ be a Borel reduction from $E$ to $F$, and let   $\sigma\colon\im(f)\to X$ be a Borel function such that $f\circ\sigma=\id_{\im(f)}$. Let $\phi\colon E\to E_\sset(A_+)$ be a Borel homomorphism such that for all $x\in X$, $\{\phi(x)(n):n\in\omega\}$ has property $\Phi$. Define the function $\tilde{\phi}\colon \im(f)\to(A_+)^\omega$ by
\[
\tilde{\phi}(y)(n)=\phi(\gamma_{n_0}\sigma(y))(n_1)\;,
\]
so that $\tilde{\phi}$ is a Borel homomorphism $F\mathbin{\upharpoonright}\im(f)\to E_\sset(A_+)$ such that for all $y\in\im(f)$, $\{\tilde{\phi}(y)(n):n\in\omega\}$ has property $\Phi$. By hypothesis and using (i), obtain a Borel homomorphism $\tilde{\psi}\colon F\mathbin{\upharpoonright}\im(f)\to E_0(A_-)$ such that for all $y\in\im(f)$ and $n\in\omega$,
  \[
  \neg(\tilde{\psi}(y)\mathrel{A}\tilde{\phi}(y)(n))\;.
  \]
  Let $\psi=\tilde{\psi}\circ f$. Then $\psi$ is a Borel homomorphism $E\to E_0(A_-)$ such that for all $x\in X$ and $n\in\omega$, $\neg(\psi(x)\mathrel{A}\phi(x)(n))$.

(iii) Let $E$ be the orbit equivalence relation arising from the Borel action of the group $\Gamma=\set{\gamma_i:i\in\omega}$ on $X$, and suppose that $\phi\colon E\to E_\sset(A_+)$ is a Borel homomorphism such that for all $x\in X$, $\{\phi(x)(n):n\in\omega\}$ has property $\Phi$. Define the function $\tilde{\phi}\colon B\to(A_+)^\omega$ by
\[\tilde{\phi}(x)(n)=\phi(\gamma_{n_0}x)(n_1)\;,
\]
so that $\tilde{\phi}$ is a Borel homomorphism $E\mathbin{\restriction}B\to E_\sset(A_+)$ such that for all $x\in B$, $\{\phi(x)(n):n\in\omega\}$ has property $\Phi$. Obtain by hypothesis a Borel homomorphism $\tilde{\psi}\colon E\mathbin{\upharpoonright}B\to E_0(A_-)$ such that for all $x\in B$ and $n\in\omega$, $\neg(\tilde{\psi}(x)\mathrel{A}\tilde{\phi}(x)(n))$. Define the function $\psi\colon X\to A_-$ by
\[
\psi(x)(n)=\tilde{\psi}(\gamma_kx)(n),\text{ where $k$ is least such that $\gamma_kx\in B$}\;.
\]
Then $\psi$ is a Borel homomorphism $E\to E_0(A_-)$ such that for all $x\in X$ and $n\in\omega$, $\neg(\psi(x)\mathrel{A}\phi(x)(n))$.

(iv) Let $f\colon X\to X$ be a Borel homomorphism from $F$ to $E$ such that $x\mathrel{F}f(x)$ for all $x\in X$. Let $B=\{x\in X: x\mathrel{E}f(x)\}$, so that $B$ is a Borel subset of $X$ that is full for $F$ such that $E\mathbin{\upharpoonright}B=F\mathbin{\upharpoonright}B$. Now if $E$ is Borel non-$(\bm{A},\Phi)$ then also $E\mathbin{\upharpoonright}B$ is Borel non-$(\bm{A},\Phi)$ by (i), which implies that $F$ is Borel non-$(\bm{A},\Phi)$ by (iii).

(v) Let $F$ be the orbit equivalence relation arising from the Borel action of the group $\Gamma=\set{\gamma_i:i\in\omega}$ on $X$. Let $\phi\colon E\to E_\sset(A_+)$ be a Borel homomorphism such that for all $x\in X$, $\{\phi(x)(n):n\in\omega\}$ has property $\Phi$. For all $x\in X$ and $n\in\omega$, define
\[\tilde{\phi}(x)(n)=\phi(\gamma_{n_0}x)(n_1)\;,
\]
so that $\tilde{\phi}$ is a Borel homomorphism $F\to E_\sset(A_+)$ such that for all $x\in X$, the set
\[
\{\tilde{\phi}(x)(n):n\in\omega\} = \bigcup_{y\in [x]_F}\{\phi(y)(n):n\in\omega\}
\]
has property $\Phi$. Obtain by hypothesis a Borel homomorphism $\psi\colon F\to E_0(A_-)$ such that for all $x\in X$ and $n\in\omega$, $\neg(\psi(x)\mathrel{A}\tilde{\phi}(x)(n))$. Then $\psi$ is a Borel homomorphism from $E$ to $E_0(A_-)$ with the property that $\neg(\psi(x)\mathrel{A}\phi(x)(n))$ for all $x,n$.
\end{proof}

Next we discuss the implication relationships between combinatorial properties derived from cardinal characteristics, which frequently correspond to inequalities between the cardinals.  Many of the inequalities expressed in van~Douwen's diagram are captured by combinatorial gadgets called \emph{generalized Galois-Tukey connections}, or \emph{morphisms} for short.  If $\bm{A}$ and $\bm{B}$ are relations, then a morphism from $\bm{A}$ to $\bm{B}$ is a pair of functions
\begin{align*}
  \xi_-&\colon B_-\to A_-\\
  \xi_+&\colon A_+\to B_+
\end{align*}
such that for all $b\in B_-$ and $a\in A_+$, 
\[
\xi_-(b)\mathrel{A}a \ \implies \ b\mathrel{B}\xi_+(a)\;.
\]
The point of the definition is that if there exists a morphism from $\bm{A}$ to $\bm{B}$, then $\|\bm{A}\|\geq\|\bm{B}\|$ (indeed, if $\mathcal F$ is dominating with respect to $\bm{A}$ then $\xi_+(\mathcal F)$ is dominating with respect to $\bm{B}$). More generally, if there exists a morphism $(\xi_-,\xi_+)$ from $\bm{A}$ to $\bm{B}$ such that whenever $\mathcal F$ has property $\Phi$ then $\xi_+(\mathcal F)$ has property $\Psi$, then $\|\bm A\|_\Phi\geq\|\bm B\|_\Psi$.

%

\begin{example} \label{ex_pleqa}
To show that $\frp\leq\fra$, one may simply observe that whenever a $\set{a_\alpha}$ is maximal almost disjoint (and infinite), then the family $\set{a_\alpha^c}$ is centered and has no pseudo-intersection. To prove this with morphisms, let $\xi_-(b)=b$ and $\xi_+(a)=a^c$; then $\xi_+$ takes infinite almost disjoint families to families with the strong finite intersection property, and whenever $b$ is not almost disjoint from $a$, we have $b\not\subset^*a^c$.
\end{example}

Although the existence of a morphism from $\bm{A}$ to $\bm{B}$ implies that $\|\bm{A}\|\geq\|\bm{B}\|$, in the study of cardinal characteristics one is ultimately interested in determining \emph{which models} of set theory satisfy a given cardinal inequality.  Of special importance are those relationships between cardinal characteristics that are provable in ZFC.  In \cite{blass-borel}, Blass notes that even if there is a morphism from $\bm{A}$ to $\bm{B}$, the inequality $\|\bm{A}\|\geq\|\bm{B}\|$ may fail in a forcing extension (with $\bm{A}$ and $\bm{B}$ interperted in the extension).  In that paper Blass proposes that one look instead for \emph{definable} morphisms, and he establishes the following.

\begin{thm}[Blass] \label{thm:forcing}
  Suppose that there is a Borel morphism from $\bm{A}$ to $\bm{B}$, meaning that the components $\xi_-,\xi_+$ are Borel functions.  Then the inequality $\|\bm A\|\geq\|\bm B\|$ cannot be violated by forcing.
\end{thm}

Borel morphisms have been studied in \cite{blass-borel}, \cite{mildenberger} and \cite{tukey}. The latter article provides a diagram of all the characteristics in Figure~\ref{fig_vand} with respect to Borel morphisms. But for our purposes, even Borel morphisms are insufficient since the existence of such a morphism from $\bm{A}$ to $\bm{B}$ does not yield a proof that Borel non-$\bm{B}$ implies Borel non-$\bm{A}$.  Moreover, it is sometimes the case that Borel non-$\bm{B}$ implies Borel non-$\bm{A}$, and yet there is no morphism from $\bm{A}$ to $\bm{B}$ at all.  For instance, we show in the next section that $\mathsf{s}\rightarrow\mathsf{r}$, even though there are models in which $\mathfrak{s}<\mathfrak{r}$.  What we need instead is the following.

\begin{defn}
  Let $\bm{A},\bm{B}$ be invariant relations.  A Borel morphism $(\xi_-,\xi_+)$ from $\bm{A}$ to $\bm{B}$ is said to be \emph{invariant} iff $\xi_-$ is a Borel homomorphism $E_0(B_-)\to E_0(A_-)$.
\end{defn}

\begin{thm}\label{thm:AbstractNonsense}
  If there exists an invariant Borel morphism from $\bm{A}$ to $\bm{B}$, then the property Borel non-$\bm{B}$ implies Borel
  non-$\bm{A}$.
\end{thm}


\begin{proof}
Let $(\xi_-,\xi_+)$ be the invariant Borel morphism from $\bm{A}$ to $\bm{B}$, and suppose that $E$ is Borel non-$\bm{B}$.  Given a Borel homomorphism $\phi\colon E\to E_\sset(A_+)$, we shall consider the homomorphism $\phi'\colon E\to E_\sset(B_+)$ defined by
\[\phi'(x)(n)=\xi_+(\phi(x)(n))\;.
\]
Since $E$ is Borel non-$\bm{B}$, there exists a Borel homomorphism $\psi'\colon E\to E_0(B_-)$ such that for all $x\in X$ and $n\in\omega$ we have $\neg(\psi'(x)\mathrel{B}\phi'(x)(n))$. Letting $\psi=\xi_-\circ\psi'$, we have that $\psi$ is a Borel homomorphism $E\to E_0(A_-)$ and for all $x\in X,n\in\omega$,
\begin{align*}
  \neg\left(\psi'(x)\mathrel{B}\phi'(x)(n)\right)
  & \implies \neg\left(\psi'(x)\mathrel{B}\xi_+(\phi(x)(n))\right) \\
  & \implies \neg\left(\xi_-(\psi'(x))\mathrel{A}\phi(x)(n)\right) \\
  & \implies \neg\left(\psi(x)\mathrel{A}\phi(x)(n)\right)\;.
\end{align*}
Hence $E$ is Borel non-$\bm{A}$.
\end{proof}

Theorem~\ref{thm:AbstractNonsense} generalizes to properties derived from arbitrary tame cardinal characteristics.  To show that property Borel non-$(\bm{B},\Psi)$ implies Borel non-$(\bm{A},\Phi)$, it suffices to find an invariant Borel morphism $(\xi_-,\xi_+)$ from $\bm{A}$ to $\bm{B}$ such that $\xi_+(\mathcal F)$ has property $\Psi$ whenever $\mathcal F$ has property $\Phi$.

We now conclude this section with a generalization of a result of Thomas which shows that under the hypothesis of Martin's Conjecture, there exists a countable Borel equivalence relation which is \emph{not} Borel bounded.  Recall that $\leq_T$ denotes the Turing reducibility relation on $2^\omega$.  The \emph{Turing equivalence relation} $\equiv_T$ defined by $x\equiv_Ty$ iff
$x\leq_Ty$ and $y\leq_Tx$ is one of the most important countable Borel equivalence relations.

Martin's Conjecture is the statement that any Borel homomorphism from $\equiv_T$ to $\equiv_T$ is either constant or $\leq_T$-increasing on a cone.  (Here, a subset of $2^\omega$ is said to be a \emph{cone} if it is $\leq_T$-upwards closed).  We shall require only the following consequence of Martin's Conjecture (for instance, see Theorem~2.1(i) in \cite{mc}):

\begin{lem}\label{lem:MC}
Assuming Martin's conjecture, if $f\colon\mathord{\equiv}_T\to E_0$ is a Borel homomorphism then there exists a cone $C$ such that $f(C)$ is contained in a single $E_0$-class.
\end{lem}

Using this, Thomas proved in Theorem~5.2 of \cite{mc} that Martin's Conjecture implies that $\equiv_T$ is not Borel bounded.  We now show that his argument applies to \emph{any} (nontrivial) property corresponding to an invariant relation.

\begin{thm}
Let $\bm{A}$ be an invariant relation, and assume that for all $z\in A_-$ there exists $a\in A_+$ such that $z\mathrel{A}a$.  Assuming Martin's Conjecture, the Turing equivalence relation $\equiv_T$ is not Borel non-$\bm{A}$.
\end{thm}

\begin{proof}
  Suppose towards a contradiction that $\equiv_T$ is Borel non-$\bm{A}$.  Let $\phi\colon 2^\omega\rightarrow A_+$ be any Borel
  function such that for all $a\in A_+$, the preimage $\phi^{-1}(a)$ is cofinal.  (To see that there exists such a map, let  $x\mapsto\seq{x_0,x_1}$ be a pairing function on $2^\omega$ and let $f$ be a Borel bijection between $2^\omega$ and $A_+$.  Then $\phi(x)= f(x_0)$ has the desired properties.)

It follows that for all $a\in A_+$, the $\equiv_T$-saturation $[\phi^{-1}(a)]_T$ contains a cone.  Since $\equiv_T$ is Borel non-$\bm{A}$, there exists a Borel homomorphism $\psi\colon E\to E_0(A_-)$ such that for all $x\in 2^\omega$,
  \[
  \neg(\psi(x)\mathrel{A}\phi(x))\;.
  \]
  By Lemma~\ref{lem:MC}, there exists a cone $C$ such that $\psi(C)$ is contained in a single $E_0(A_-)$-class, say $[z]_{E_0}$.  Since $\bm{A}$ is nontrivial, there exists $a\in A_+$ such that $z\mathrel{A}a$.  Since $[\phi^{-1}(a)]_T$ contains a cone, it meets $C$.  In particular, there exists $x\in 2^\omega$ such that $\psi(x)\mathrel{A}\phi(x)$, which is a contradiction.
\end{proof}

\section{The non-splitting property}

The non-splitting property $\sfs$ would appear to be very special, as it does not hold of any non-trivial equivalence relations.

\begin{thm} \label{thm_splitting}
  If the countable Borel equivalence relation $E$ has property $\sfs$, then $E$ is smooth.
\end{thm}

For the proof we shall require the following standard measure-theoretic fact.  Let $m$ denote the Haar (or ``coin-tossing'') measure on $2^\omega$, so that $m$ is the $\omega$-fold product of the $(\frac12,\frac12)$ measure on $\set{0,1}$. Then $[\omega]^\omega$ is an $E_0$-invariant, Borel, $m$-conull subset of $2^\omega$, and we denote the restriction of $m$ to $[\omega]^\omega$ also by $m$.

\begin{prop} \label{prop_nonmeas}
  Suppose that $N\subset [\omega]^\omega$ has the following properties:
  \begin{enumerate}
  \item $N$ is $E_0$-invariant;
  \item for any $x\in [\omega]^\omega$, exactly one of $x$ or $x^c$ is in $N$.
  \end{enumerate}
  Then $N$ is not $m$-measurable.
\end{prop}

\begin{proof}
Suppose that $N$ is measurable. Property (a) says that $N$ is a tail event, so by Kolmogorov's zero-one law $N$ has measure zero or one. On the other hand, property (b) implies that $N$ has measure $\frac12$, since the map $x\mapsto x^c$ is a measure-preserving bijection which sends $N$ onto $[\omega]^\omega\smallsetminus N$.
\end{proof}


\begin{proof}[Proof of Theorem~\ref{thm_splitting}]
  By Theorem~\ref{thm:closure}, it suffices to show that there exists a hyperfinite equivalence relation $E$ which does not have property $\sfs$. To this end, let $E$ be the hyperfinite equivalence relation on $X=[\omega]^\omega$ given by
  \[
  x\mathrel{E}x'\iff x\mathrel{E}_0x'\text{\ \ or\ \ } x^c\mathrel{E}_0x'\;,
  \]
  and assume that $E$ is the orbit equivalence relation arising from the Borel action of the countable group   $\Gamma=\set{\gamma_i:i\in\omega}$ on $X$.

Now suppose that $E$ is Borel non-splitting, and define the Borel homomorphism $\phi\colon E\to E_\sset$ by
\[
\phi(x)(n)=\gamma_n x\;,
\]
so that each $\phi(x)$ enumerates the family $[x]_E$. Let $\psi\colon E\to E_0$ be a Borel homomorphism such that for all $x$ and $n$,
\[
\text{either}\quad\psi(x)\subset^*\phi(x)(n)\quad\text{or}\quad\psi(x)\subset^*(\phi(x)(n))^c\;.
\]
Then in particular, for each $x$ we have either $\psi(x)\subset^*x$ or $\psi(x)\subset^*x^c$.  Now put
\[
N=\set{x\in X:\psi(x)\subset^*x}\;.
\]
Since $\psi$ is Borel, so is $N$.

To complete the proof, we show that $N$ satisfies conditions (a) and (b) of Proposition~\ref{prop_nonmeas}, and therefore is not $m$-measurable, a contradiction.  For (a), suppose that $x\in N$, so that $\psi(x)\subset^*x$, and let $x'$ be such that $x'\mathrel{E}_0x$.  Then $\psi(x')\mathrel{E}_0\psi(x)$, so that $\psi(x')=^*\psi(x)\subset^*x=^*x'$, and hence $x'\in N$ too. For condition (b), suppose that $x\in N$ so that $\psi(x)\subset^*x$.  Then since $x^c\mathrel{E}x$, we have $\psi(x^c)\mathrel{E}_0\psi(x)$, so that $\psi(x^c)=^*\psi(x)\subset^*x$.  It follows that $\psi(x^c)\not\subset^*x^c$, and hence $x^c\not\in N$.  This shows that $x\in N$ implies $x^c\in 2^\omega\smallsetminus N$, and the converse is the same.
\end{proof}

%
%
%

We remark that as a consequence, we can also obtain a similar result for the properties derived from the cardinal characteristics $\mathfrak{par}_n$ discussed in \cite[Section~3]{blass}. In detail, let us say that $E$ \emph{has property} $\sfpar_n$ if for every Borel homomorphism $\phi\colon E\to E_\sset(2^{\omega^n})$, there exists a Borel homomorphism $\psi\colon E\to E_0([\omega]^\omega)$ such that for all $x\in X$, $\psi(x)$ is almost homogeneous for each function $\phi(x)(n)$.  (Here, $A\subset\omega$ is \emph{almost homogeneous} for a function $f$ if $A$ is almost equal to a set which is homogeneous for $f$.) It is easy to see that property $\sfpar_1$ coincides with property $\sfs$, and $\sfpar_n\rightarrow\sfpar_m$ for $n\geq m$. It follows that if $E$ has property $\sfpar_n$, then $E$ is smooth.

We close this section by noting that the converse of Theorem 6.1 also holds. Indeed, each of the properties in Definition~\ref{df:list}
holds of smooth relations. For properties other than $\sfs$, this is implicit in Proposition~\ref{prop:p} together with Theorems~\ref{thm:diagram} and \ref{thm:closure}, but intuitively, if $E$ is a smooth countable Borel equivalence relation on the standard Borel space $X$ with Borel transversal $B\subset X$, then for a given property $\mathsf{x}$ corresponding to the cardinal $\frx$, we simply use the unique point in $[x]_E\cap B$ as the base point for the appropriate diagonalization argument showing that
$\set{\phi(x)(n):n\in\omega}$ is not dominating with respect to $\mathfrak{x}$. This is perhaps hardest to carry out in the case of $\sfs$, so for completeness we do so here.

Thus suppose $E$ is a smooth countable Borel equivalence relation on the standard Borel space $X$, with Borel transversal $B\subset X$ and Borel function $\sigma\colon X\rightarrow X$ such that for each $x\in X$, $\sigma(x)$ is the unique element in $B\cap [x]_E$. Let $\phi\colon E\rightarrow E_\sset([\omega]^\omega)$ be a given Borel homomorphism. Inductively define the Borel functions $\phi_n'\colon X\rightarrow [\omega]^\omega$ by $\phi_0'(x)=\phi(\sigma(x))(0)$ and
\begin{equation*}
  \phi_{n+1}'(x)=
  \begin{cases}
    (\phi\circ\sigma)(x)(n+1)\,\cap\,\phi'(x)(n)
      & \text{if this set is infinite\;;} \\
    (\phi\circ\sigma)(x)(n+1)^c\,\cap\,\phi'(x)(n)
      & \text{otherwise}\;.
  \end{cases}
\end{equation*}
Then inductively define the Borel functions $\alpha_n\colon X\rightarrow\omega$ by
\[
\alpha_n(x)= \min\left[\phi_n'(x)\smallsetminus\set{\alpha_k(x):k<n}\right]\;,
\]
and finally set $\psi(x)(n)=\alpha_n(x)$. It is easily checked that $\psi(x)\subset^*\phi(x)(n)$ or $\psi(x)\subset^*\phi(x)(n)^c$ for each $x\in X$ and $n\in\omega$. Moreover, $\psi$ is not merely a Borel homomorphism $E\to E_0([\omega]^\omega)$, but in fact is $E$-invariant.

\section{Proof of the diagram}

In this section, we prove the implications between the combinatorial properties depicted in Figure~\ref{fig_invar}.

\begin{thm}
  \label{thm:diagram}
  The arrows in Figure~\ref{fig_invar} correspond to true implications between properties of countable Borel equivalence relations.
\end{thm}

The first implication, $\sfs\rightarrow\sfp$, follows from Theorem~\ref{thm_splitting} together with the fact that smooth equivalence relations have property $\sfp$.  In fact, this implication is not reversible, since there also exist nonsmooth relations with property $\sfp$.

\begin{prop} \label{prop:p}
  If $E$ is hyperfinite, then $E$ has property $\sfp$.
\end{prop}

\begin{proof}
  Express $E$ as the increasing union of finite Borel equivalence relations $F_n$, and suppose we are given a Borel homomorphism $\phi\colon E\to E_\sset([\omega]^\omega)$ such that for each $x\in X$, the family $\set{\phi(x)(n):n\in\omega}$ is centered. Set $\alpha_0(x)=0$ for all $x\in X$, and inductively define functions $a_n\colon X\to\omega$ as follows.  Given $a_0(x),\ldots,a_n(x)$, let
  \[
  a_{n+1}(x)=\min\left(\bigcap\{\phi(y)(i) : i\leq n \text{ and }x\mathrel{F}_ny\} \smallsetminus \set{a_i(y):i\leq n\text{ and } x\mathrel{F}_ny}\right)\;.
  \]
  This construction is just a slight reorganizing of the usual diagonalization, and if we let $\psi(x)=\set{a_n(x):n\in\omega}$, then $\psi(x)$ is a pseudo-intersection of the set $\set{\phi(x)(n):n\in\omega}$.  Moreover, the $a_n$ have the property that if $x\mathrel{E}y$, then the sequences $\seq{a_n(x):n\in\omega}$ and $\seq{a_n(y):n\in\omega}$ will eventually be equal, so that $\psi$ is also a homomorphism from $E$ to $E_0([\omega]^\omega)$.
\end{proof}

Currently, $\sfs\rightarrow\sfp$ and its consequences are the only implications that we can prove are nonreversible.

The remainder of the proof of Theorem~\ref{thm:diagram} will be given in a series of lemmas.  In most cases the proofs amount to trivial observations, such as noticing that the standard proof of the corresponding cardinal inequality can be carried out in a Borel and invariant fashion.  However, in a few cases some care is needed to make sure that this can be done.

\begin{lem}
  $\sfp\rightarrow\sft$.
\end{lem}

\begin{proof}
  This holds simply because every tower is a centered family.
\end{proof}

\begin{lem}
  $\sfp\rightarrow\sfa$.
\end{lem}

\begin{proof}
  It suffices to observe that the morphism $\xi_-(b)=b$, $\xi_+(a)=a^c$ described in Example~\ref{ex_pleqa} is Borel and invariant.
\end{proof}

\begin{lem} \label{lem:pb}
  $\sfp\rightarrow\sfb$.
\end{lem}

This is a difficult case in which the classical proof that $\frp\leq\frb$ apparently cannot be carried out in an invariant fashion.  We were able to obtain only the weaker result $\sfp\rightarrow\sfu$, and the problem remained open until Tam\'as M\'atrai and Juris Stepr\=ans provided us with a positive answer. Since the proof appears in \cite{tukey}, we give only a brief outline here.

To establish $\sfp\rightarrow\sfb$ with a morphism, we require maps $\xi_-\colon[\omega]^\omega\to\omega^\omega$, $\xi_+\colon\omega^\omega\to[\omega]^\omega$ such that $\im(\xi_+)$ is centered and
\[
A\subset^*\xi_+(f)\implies f\leq^*\xi_-(A)\;.
\]
To do this one constructs the map $\xi_+$ with the additional property that for every $\leq^*$-unbounded subset $S\subset\omega^\omega$, the set $\xi_+(S)$ does not have a pseudo-intersection.  It follows that for each $A\in[\omega]^\omega$, the set $S_A=\set{f\in\omega^\omega:A\subset^*\xi_+(f)}$ is $\leq^*$-bounded.  Letting $\xi_-(A)$ be such a bound, it is easy to see that $\xi_-,\xi_+$ satisfy the required properties.  Finally, it is possible to compute the bounds $\xi_-(A)$ in a Borel fashion.\footnote{In this construction, $\xi_-$ will be $E_0$-invariant.  In fact we can obtain this extra property automatically by the following general observation: If there is a Borel morphism from $\bm{A}$ to $\bm{B}$, $\neg A$ is transitive, and $E_0$ is Borel non-$\bm{A}$, then there is an \emph{invariant} Borel morphism from $\bm{A}$ to $\bm{B}$.}

\begin{lem}
  $\sfb\rightarrow\sfd$.
\end{lem}

\begin{proof}
  This holds simply using the identity morphism, because any function which bounds a given family also witnesses that the family is not dominating.
\end{proof}

\begin{lem}
  $\sfb\rightarrow\sfr$.
\end{lem}

\begin{proof}
  We follow the proof that $\frb\leq\frr$ that is given in \cite{blass}.  Given a subset $x\subset\omega$, let $f_x\colon\omega\to\omega$ be a function such that each interval $[n,f_x(n))$ contains an element of $x$.  Now, given a Borel homomorphism $\phi\colon E\rightarrow E_\sset([\omega]^\omega)$, apply property $\sfb$ to the function $\tilde\phi(x)(n)=f_{\phi(x)(n)}$ to obtain a Borel homomorphism $\psi\colon E\to E_0([\omega]^\omega)$ such that for all $x$ and $n$, $\psi(x)$ eventually dominates $f_{\phi(x)(n)}$.

  \begin{claim*}
    There exists a Borel homomorphism $\lambda\colon E_0(\omega^\omega)\to E_0(\omega^\omega)$ such that for every $f\in\omega^\omega$, every interval $[\lambda(f)(j),\lambda(f)(j+1))$ contains an interval of the form $[n,f(n))$.
  \end{claim*}

  \begin{claimproof}
    Begin by expressing $=^*$ as an increasing union of finite Borel equivalence relations $F_i$.  We inductively define an     increasing sequence of functions $j_i\colon\omega^\omega\to\omega$ as follows.  Given $j_i$, define an auxiliary function $k_{i+1}$ so that for all $f$, the interval $[j_i(f),k_{i+1}(f))$ contains an interval of the form $[n,f(n))$.  Then, let 
    \[
    j_{i+1}(f)=\max\set{k_{i+1}(g):g\mathrel{F_i}f}\;.
    \]
    We now let $\lambda(f)=\seq{j_i(f)}$, and it is clear that $\lambda$ is as desired.
  \end{claimproof}

  We now consider the composition $\lambda\circ\psi$, which has the property that for all $m$, almost every $[\lambda\circ\psi(x)(n),\lambda\circ\psi(x)(n+1))$ contains an element of $\phi(x)(m)$.  We may therefore define
  \[
  \psi'(x)=\bigcup_{\textrm{$n$ odd}}[(\lambda\circ\psi)(x)(n),(\lambda\circ\psi)(x)(n+1))\;,
  \]
  and we will have that $\psi'(x)$ splits each member of the family enumerated by $\phi(x)$.
\end{proof}

\begin{lem}
  $\sfr\rightarrow\sfu$
\end{lem}

\begin{proof}
  Once again the identity morphism suffices, because if some set $Y$ witnesses that a family $\mathcal F$ is not unsplittable, then it also witnesses that $\mathcal F$ is not an ultrafilter base. (Property $\sfr$ and $\sfu$ are the same except $\sfu$ has a centered hypothesis on the family.)
\end{proof}

\begin{lem}
  $\sfr\rightarrow\sfi$.
\end{lem}

\begin{proof}
  Suppose the Borel homomorphism $\phi\colon E\to E_\sset$ has the property that every $\phi(x)$ enumerates an independent family.  Let $\phi'(x)$ enumerate all possible intersections of finitely many sets from $\phi(x)$ with the complements of finitely many other sets from $\phi(x)$.  Since $E$ has property $\sfr$, there exists a Borel homomorphism $\psi\colon E\to E_0$ such that for every $x$ and $n\in\omega$, $\psi(x)$ splits $\phi'(x)(n)$.  Then $\psi(x)$ also witnesses that the family $\phi(x)$ is not maximal independent.
\end{proof}

This completes the proof of Theorem~\ref{thm:diagram}.

\bibliographystyle{alpha}
\begin{singlespace}
     \bibliography{char-cber}
\end{singlespace}

\end{document}